\documentclass[11pt,a4paper]{article}

\usepackage{latexsym}
\usepackage{amssymb}
\usepackage{amsthm}
\usepackage{amscd}
\usepackage{amsmath}
\usepackage[all]{xy}
\usepackage{graphicx}

\newtheorem{theorem}{Theorem}[section]

\newtheorem{lemma}{Lemma}[section]

\newtheorem{corollary}{Corollary}[section]

\xyoption{dvips}

\numberwithin{equation}{section}

\newcommand{\FF}{\mathbb{F}}    \newcommand{\CC}{\mathbb{C}}    
 \def\dim{\mathrm{dim}}  
 \def\ZZ{\mathbb{Z}} 
   \def\Ind{\mathrm{Ind}} \def\GL{\mathrm{GL}}   \def\Res{\mathrm{Res}} \def\dd{\displaystyle}   \def\spanning{\textnormal{-span}}
\def\vphi{\varphi}

\newcommand{\fkn}{\mathfrak{n}}

\newcommand{\Tr}{\mathrm{Tr}}
\newcommand{\cP}{\mathcal{P}}

\newcommand{\cS}{\mathcal{S}}

\newcommand{\One}{{1\hspace{-.14cm} 1}}

\newcommand{\SInd}{\mathrm{SInd}}
\newcommand{\SRes}{\mathrm{SRes}}
\newcommand{\larc}[1]{\hspace{-.4ex}\overset{#1}{\frown}\hspace{-.4ex}}
\newcommand{\cR}{\mathcal{R}}

\makeatletter
\renewcommand{\@makefnmark}{\mbox{\textsuperscript{}}}
\makeatother

\allowdisplaybreaks[1]

\UseCrayolaColors

\def\adots{\mathinner{\mkern2mu\raise0pt\hbox{.}  
\mkern2mu\raise4pt\hbox{.}\mkern1mu
\raise7pt\vbox{\kern7pt\hbox{.}}\mkern1mu}}

\begin{document}

\title{Superinduction for pattern groups}
\author{Eric Marberg\footnote{Stanford University: \textsf{emarberg@stanford.edu}}{ }  and Nathaniel Thiem\footnote{University of Colorado at Boulder: \textsf{thiemn@colorado.edu}}}
\date{}

\maketitle

\abstract{It is well-known that the representation theory of the finite group of unipotent 
upper-triangular matrices $U_n$ over a finite field is a wild problem.  By instead 
considering approximately irreducible representations (supercharacters), one obtains a rich combinatorial 
theory analogous to that of the symmetric group, where we replace partition combinatorics 
with set-partitions.    This paper studies Diaconis--Isaacs' concept of superinduction in pattern groups.   While superinduction shares many desirable properties with usual induction, it no longer takes characters to characters.  We begin by finding sufficient conditions guaranteeing that super-induction is in fact induction.  It turns out for natural embedding of $U_m$ in $U_n$, super-induction is induction.   We conclude with an explicit combinatorial algorithm for computing this induction 
analogous to the Pieri-formulas for the symmetric group.}

\section{Introduction}

Understanding the representation theory of the finite group of upper-triangular matrices $U_n$ is  a well-known wild problem.  Therefore, it came as somewhat of a surprise when C. Andr\'e was able to show that by merely ``clumping" together some of the conjugacy classes and some of the irreducible representations one attains a workable approximation to the representation theory of $U_n$ \cite{An95,An99,An01,An02}. In his Ph.D. thesis \cite{Ya01}, N. Yan showed how the algebraic geometry of the original construction could be replaced by more elementary constructions.   E. Arias-Castro, P. Diaconis, and R. Stanley \cite{ADS04} then demonstrated that this theory can in fact be used to study random walks on $U_n$ using techniques that traditionally required the knowledge of the full character theory \cite{DS93}.  Thus, the approximation is fine enough to be useful, but coarse enough to be computable.  Furthermore, this approximation has a remarkable combinatorial structure analogous to that of the symmetric group, where we replace partitions with set-partitions,
$$\left\{\begin{array}{c} \text{Almost irreducible}\\ \text{representations of $U_n$}\end{array}\right\} \longleftrightarrow \left\{\begin{array}{c} \text{Labeled set partitions}\\ \text{of $\{1,2,\ldots,n\}$}\end{array}\right\}.$$
 One of the main results of this paper is to extend the analogy with the symmetric group by giving a combinatorial Pieri-like formula for set-partitions that corresponds to induction in $U_n$.

In \cite{DI06}, P. Diaconis and M. Isaacs generalized this approximating approach to develop a the concept of a \emph{supercharacter theory} for all finite groups, where irreducible characters are replaced by supercharacters and conjugacy classes are replaced by superclasses.  In particular, their paper generalized Andr\'e's original construction by giving an example of a supercharacter theory for a family of groups called algebra groups.  For this family of groups, they show that supercharacters restrict to $\ZZ_{\geq 0}$-linear combination of supercharacters, tensor products of supercharacters are $\ZZ_{\geq 0}$-linear combinations of supercharacters, and they develop a notion of superinduction that is the adjoint functor to restriction for supercharacters.   Unfortunately, a superinduced supercharacter is not necessarily a $\ZZ_{\geq 0}$-linear combination of supercharacters; in fact, it need not be a character at all.   This paper examines superinduction more closely, giving sufficient conditions for when superinduction is induction (in which case are guaranteed characters).   

Section 2 reviews the notion of a supercharacter theory, and defines the fundamental combinatorial and algebraic objects needed for the main results.   Section 3 proceeds to compare superinduction to induction, where the main results --Theorems \ref{AlgebraGroupSuperinductionIsInduction}, \ref{SemidirectProductSuperinduction}, and \ref{CosetRepresentativesSuperinduction} -- give sufficient conditions for when superinduction is induction.   As a consequence, we conclude that superinduction between $U_m\subseteq U_n$ is induction.  Section 4 decomposes induced supercharacters from $U_m$ to $U_n$ as combinatorial ``products" on set-partitions $\mu$ of the form
$$\Ind_{U_m}^{U_n}(\chi^\mu)=\chi^\mu\ast_1 \chi^{\{m+1\}}\ast_1 \chi^{\{m+2\}}\ast_1\cdots \ast_1 \chi^{\{n\}}.$$  

This paper is a companion paper to \cite{TVe07}, which studies the restriction of supercharacters by analyzing a family of subgroups that interpolate between $U_n$ and $U_{n-1}$. Other work related to supercharacter theory of unipotent groups include C. Andr\'e and A. Neto's  exploration of supercharacter theories for unipotent groups of Lie types $B$, $C$, and $D$ \cite{AN06}, C. Andr\'e and  A. Nicol\'as' analysis of supertheories over other rings \cite{AnNi06}, and an intriguing possible connection between supercharacter theories and  Boyarchenko and Drinfeld's work on $L$-packets \cite{BD06}.

\subsubsection*{Acknowledgements}

We would like to thank Diaconis and Venkateswaran for many enlightening discussions regarding this work, and the first author would like to thank Stanford University for summer funding during parts of this research.

\section{Preliminaries}

This section reviews the concept of a supercharacter theory, defines a family of groups called pattern groups, and establishes the combinatorial notation for labeled set-partitions.

\subsection{Supercharacter theory}
 
Let $G$ be a group.  A \emph{supercharacter theory} for $G$ is a partition $\cS^\vee$ of the elements of $G$ and a set of characters $\cS$, such that
\begin{enumerate}
\item[(a)] $|\cS|=|\cS^\vee|$,
\item[(b)] Each $S\in \cS^\vee$ is a union of conjugacy classes,
\item[(c)] For each irreducible character $\gamma$ of $G$, there exists a unique $\chi\in \cS$ such that
$$\langle \gamma, \chi\rangle>0,$$
where $\langle,\rangle$ is the usual inner-product on class functions,
\item[(d)] Every $\chi\in \cS$ is constant on the elements of $\cS^\vee$.
\end{enumerate}
We call $\cS^\vee$ the set of \emph{superclasses} and $\cS$ the set of \emph{supercharacters}.  Note that every group has two trivial supercharacter theories -- the usual character theory and the supercharacter theory with $\cS^\vee=\{\{1\},G\setminus\{1\}\}$ and $\cS=\{\One,\gamma_G-\One\}$, where $\One$ is the trivial character of $G$ and $\gamma_G$ is the regular character.

There are many ways to construct supercharacter theories, but this paper will study a particular version developed in \cite{DI06} to generalize Andr\'e's original construction to a larger family of groups called algebra groups.

\subsection{Pattern groups}\label{PatternGroups}

While many results can be stated in the generality of algebra groups, many statements become simpler if we restrict our attention to a subfamily called pattern groups.

Let $U_n$ denote the set of $n\times n$ unipotent upper-triangular matrices with entries in the finite field $\FF_q$ of $q$ elements.   Let $\cP$ be a poset on the set $\{1,2,\ldots, n\}$.  The pattern group $U_\cP$ is given by
$$U_\cP=\{u\in U_n\ \mid\ u_{ij}\neq 0\text { implies $i<j$ in $\cP$}\}.$$
Examples of groups in this family include all the unipotent radicals of rational parabolic subgroups of the finite general linear groups $\GL_n(\FF_q)$, of which $U_n$ is the pattern group corresponding to the total order $1<2<3<\cdots<n$.  

An obvious advantage to pattern groups is the associated poset structure, which can be used to describe a variety of group theoretic structures, such as the center, the Frattini subgroups, etc.  We can also describe coset representatives using the poset structure as follows. 

 Let $U_{\cP}\subseteq U_{\cR}$ be pattern groups with corresponding posets $\cP$ and $\cR$.  Let $\cR/\cP$ be the set of relations given by 
 $$i\leq j \text{ in } \cR/\cP,\qquad\text{if $i\leq j$ in $\cR$ and $i\nleq j$ in $\cP$.}$$
 Note that $\cR/\cP$ is not necessarily a poset.   

\begin{lemma} \label{LeftRightCosetRepresentatives}
Let $U_\cP\subseteq U_{\cR}$ be pattern groups.  Then
$$I=\{u\in U_{\cR}\ \mid\  u_{ij}\neq 0 \text{ implies $i<j$ in $\cR/\cP$}\}$$
is a set of left coset representatives for $U_{\cR}/U_{\cP}$ and a set of right coset representatives for $U_{\cP}\backslash U_{\cR}$.
\end{lemma}
\begin{proof}
For $i<j$, let
$$e_{ij}=\text{the $n\times n$ matrix with 1 in position $(i,j)$ and zeroes elsewhere.}$$
Given any total order $\mathcal{T}$ on $\{(i,j)\ \mid\ 1\leq i<j\leq n\}$ and $u\in U_n$, there exist unique coefficients $t_{ij}\in \FF_q$ such that
$$u=\prod_{i<j} (1+t_{ij}e_{ij}),$$
where the product is according to the total order $\mathcal{T}$.  

Fix the total order on $\{(i,j)\ \mid\ 1\leq i<j\leq n\}$ given by
$$(k,l)\leq (i,j), \qquad \text{if $i<k$ OR if $i=k$ and $j\leq l$.}$$
For example, the ordering on $\{(i,j)\ \mid\ 1\leq i<j\leq 4\}$ is
$$(3,4)<(2,4)<(2,3)<(1,4)<(1,3)<(1,2).$$

If $u\in U_\cR$, then
\begin{equation}\label{LeftCosetDecompositionEquation}
u=\prod_{i<j\text{ in $\cP$}} (1+t_{ij} e_{ij}) \prod_{i<j\text{ in $\cR/\cP$}} (1+t_{ij} e_{ij}),
\end{equation}
where $t_{ij}\in \FF_q$ and each product is taken individually according to the fixed total order (That is, we are additionally taking all pairs in $\cR/\cP$ to be greater than those in $\cP$).  Thus,
$$\prod_{i<j\text{ in $\cR/\cP$}} (1+t_{ij} e_{ij})\in L$$
is in the same right coset as $u$.  Furthermore, since (\ref{LeftCosetDecompositionEquation}) is unique, two elements of $I$ cannot be in the same right coset.  The argument for left coset representatives is similar.
\end{proof}

\subsection{Superclasses}

The group $U_\cP$ has a two-sided action on the $\FF_q$-algebra
$$\fkn_\cP=\{u-1\ \mid\ u\in U_\cP\},$$
by both left and right multiplication.  Two elements $u,v\in U_\cP$ are in the same \emph{superclass} if $u-1$ and $v-1$ are in the same two-sided orbit in $\fkn_\cP$.  

\vspace{.25cm}

\noindent\textbf{Superclass row and column reducing.}   Note that since every element of $U_\cP$ can be decomposed as a product of elementary matrices, every element in the orbit containing $v-1\in\fkn_\cP$ can be obtained by applying a sequence of the following row and column operations.
\begin{enumerate}
\item[(a)] A scalar multiple of row $j$ may be added to row $i$ if $j>i$ in $\cP$,
\item[(b)] A scalar multiple of column $k$ may be added to column $l$ if $k<l$ in $\cP$.
\end{enumerate}

\vspace{.25cm}

\noindent\textbf{Example.}  Let
\begin{equation}\label{ExamplePoset}
\cP=\xy<0cm,.8cm>\xymatrix@R=.4cm@C=.25cm{ *={} & *{5} \\
*={} & *{4}\ar @{-} [u] \\ *{1} \ar @{-} [ur] & *{3}  \ar @{-} [u]   \\ & *{2}\ar @{-} [u]   }\endxy\ .
\end{equation}
Then the matrices
$$\left\{\left(\begin{array}{ccccc} 
1 & 0& 0 & 1 & a \\ 
0 & 1 & 0 & 1 & b\\
0 & 0 & 1 & 0 & c\\
0 & 0 & 0 & 1 & 1\\
0 & 0 & 0 & 0 & 1\end{array}\right)\ \bigg|\ a,b,c,d\in \FF_q\right\}\subseteq U_\cP$$
form a superclass of $U_\cP$.

\subsection{Supercharacters}\label{SectionPatternSupercharacters}

The group $U_\cP$ has a two-sided action on the dual space
$$\fkn_\cP^*=\{\lambda:\fkn\rightarrow \FF_q\ \mid\ \lambda\text{ $\FF_q$-linear}\}$$
given by
$$(u\lambda v)(x-1)=\lambda(u^{-1}(x-1)v^{-1}), \qquad\text{where $\lambda\in \fkn_\cP^*$, $u,v,x\in U_\cP$.}$$

Fix a nontrivial group homomorphism $\theta:\FF_q^+\rightarrow \CC^\times$.  The \emph{supercharacter} $\chi^\lambda$ corresponding to $\lambda\in \fkn_\cP^*$ is given by
$$\chi^\lambda(u)=\frac{|U_\cP \lambda|}{|U_\cP \lambda U_\cP|}\sum_{\mu\in U_\cP(-\lambda) U_\cP} \theta\circ \mu(u-1), \qquad\text{for $u\in U_\cP$}.$$
The corresponding modules $V^{-\lambda}$ are given by
$$V^{-\lambda}=\CC\spanning\{v_\mu\ \mid\ \mu\in U_\cP(-\lambda)\},$$
with $U_\cP$ action given by
$$uv_\mu=\theta(\mu(u^{-1}-1))v_{u\mu},\qquad\text{where $u\in U_\cP$, $\mu\in U_\cP(-\lambda)$.}$$

\vspace{.25cm}

\noindent\textbf{Remark.} Given $\lambda\in \fkn_\cP$, there are two natural choices for the supercharacters $\chi^\lambda$ -- whether to sum over $U_\cP \lambda U_\cP$ or $U_\cP (-\lambda) U_\cP$.  We use the convention of \cite{DI06} rather than the one in \cite{DT07}.  With this choice, the character formulas are slightly simpler.

\vspace{.25cm}

\noindent\textbf{Supercharacter row and column reducing.}
We many identify  the functions $\lambda\in \fkn_\cP^*$ with the matrix $\lambda\in \fkn_\cP$ by the convention
$$\lambda_{ij}=\lambda(e_{ij}),\qquad \text{where $e_{ij}\in \fkn_\cP$ has 1 in position $(i,j)$ and zeroes elsewhere.}$$
Then for $u,v\in U_\cP$ and $\lambda\in \fkn_\cP^*$,
$$(u^{-1}\lambda v^{-1})(e_{ij})=\lambda(ue_{ij}v)=\sum_{1\leq k,l\leq n} u_{ki}\lambda_{kl} v_{jl}=\big(\Tr(u) \lambda \Tr(v)\big)_{ij},$$
where $\Tr(u)$ is the transpose of the matrix $u$.  Thus, we can compute the function $(u\lambda v)$ directly by computing the matrix $\Tr(u^{-1})\lambda\Tr(v^{-1})$ and setting all entries to zero that cannot be nonzero in $\fkn_\cP$.  Alternatively, we can apply a sequence of column operations,
\begin{enumerate}
\item[(a)] A scalar multiple of row $i$ may be added to row $j$ if $i<j$ in $\cP$,
\item[(b)] A scalar multiple of column $l$ may be added to column $k$ if $l>k$ in $\cP$,
\end{enumerate}
where we set to zero all nonzero entries that might occur through these operations that are not in allowable in $\fkn_\cP$.

\vspace{.25cm}

\noindent\textbf{Example.} Let $\cP$ be as in (\ref{ExamplePoset}).  Then the set of matrices
$$\left\{\left(\begin{array}{ccccc} 
0 & 0& 0& a & 1 \\ 
0 & 0 & b & a & 1\\
0 & 0 & 0 & ac & c\\
0 & 0 & 0 & 0 & d\\
0 & 0 & 0 & 0 & 0\end{array}\right)\ \bigg|\ a,b,c,d\in \FF_q\right\}\subseteq \fkn_\cP,$$
all give rise to the same supercharacter of $U_\cP$.

\subsection{$\FF_q$-labeled set-partitions}

A fundamental example is the group of unipotent upper-triangular matrices $U_n$, corresponding to the total order $\cP=\{1<2<\cdots<n\}$.  By row and column reducing as in Section \ref{PatternGroups}, we may choose superclass representatives for $U_n$ so that
\begin{equation} \label{UOrbitRepresentatives}
\left\{\begin{array}{c} \text{Superclasses}\\ \text{of $U_n$}\end{array}\right\}\longleftrightarrow 
\left\{u\in U_n\ \big|\ \begin{array}{c} \text{$u-1$ has at most one nonzero}\\ \text{ element in every and column}\end{array}\right\}
\end{equation}
Similarly, character representatives of $U_n$ also have a nice description.
\begin{equation} \label{UcoOrbitRepresentatives}
\left\{\begin{array}{c} \text{Supercharacters}\\ \text{of $U_n$}\end{array}\right\}\longleftrightarrow 
\left\{\lambda\in \fkn_n\ \big|\ \begin{array}{c} \text{$\lambda$ has at most one nonzero}\\ \text{ element in every and column}\end{array}\right\}
\end{equation}

These representatives have a combinatorial description as labeled set partitions.  A \emph{set partition} $\lambda=\{\lambda_1\mid\lambda_2\mid\ldots\mid\lambda_\ell\}$ of $\{1,2,\ldots, n\}$ is a collection of pairwise disjoint, increasing subsequences  such that $\{1,2,\ldots,n\}=\lambda_1\cup \lambda_2\cup \cdots\cup \lambda_\ell$; the subsequences $\lambda_j$ are called the \emph{parts} of the set partition.  We will order the subsequences by their smallest elements.

\vspace{.25cm}

\noindent \textbf{Example.}  The set partitions of $\{1,2,3\}$ are
$$\{1\larc{}2\larc{}3\},\quad \{1\larc{}2\mid 3\},\quad \{1\larc{}3\mid 2\},\quad \{1\mid 2\larc{}3\},\quad \{1\mid 2\mid 3\}.$$

If $q=2$, then 
$$\left\{u\in U_n\ \big|\ \begin{array}{c} \text{$u-1$ has at most one nonzero}\\ \text{ element in every and column}\end{array}\right\} \longleftrightarrow \left\{\begin{array}{c}\text{Set partitions}\\ \text{of $\{1,2,\ldots,n\}$}\end{array}\right\},$$
where given a representative $u$, we construct a set partition by the rule
$$i\larc{} j \qquad \text{if and only if}\qquad u_{ij}=1.$$
For example,
$$u-1=\left(\begin{array}{cccc} 0 & 1 & 0 & 0\\ 0 & 0 & 0 & 1\\ 0 & 0 & 0 & 0\\ 0 & 0 & 0 & 0\end{array}\right)\longleftrightarrow \{1\larc{}2\larc{}4\mid 3\}.$$

Labeled set partitions are the combinatorial generalization that arise when $q>2$.  An $\FF_q$\emph{-labeled set partition} is a set partition $\lambda$ where we label each pair of adjacent elements $i\larc{}j$ with some nonzero element $s\in \FF_q^\times$.

\vspace{.25cm}

\noindent\textbf{Example.}  The $\FF_q$-labeled set partitions of $\{1,2,3\}$ are
$$\{1\larc{s}2\larc{t}3\},\quad \{1\larc{s}2\mid 3\},\quad \{1\larc{s}3\mid 2\},\quad \{1\mid 2\larc{t}3\},\quad \{1\mid 2\mid 3\},\qquad \text{where $s,t\in \FF_q^\times$.}$$

With this notation, for all $q$, 
$$\left\{u\in U_n(\FF_q)\ \big|\ \begin{array}{c} \text{$u-1$ has at most one nonzero}\\ \text{ element in every and column}\end{array}\right\} \longleftrightarrow \left\{\begin{array}{c}\text{$\FF_q$-labeled set partitions}\\ \text{of $\{1,2,\ldots,n\}$}\end{array}\right\}.$$

\vspace{.25cm}

Let
\begin{equation}
\cS_n(q)=\{\text{$\FF_q$-labeled set partitions of $\{1,2,\ldots, n\}$}\}.
\end{equation}

\section{Superinduction}

In \cite{DI06}, Diaconis and Isaacs introduce a notion of superinduction, as a dual functor to restriction.  Specifically, if $H\subseteq G$ are algebra groups and $\chi$ is a superclass function of $H$, then for $g\in G$,
\begin{equation}\label{SuperinductionDefinition}
\SInd_H^G(g)=\frac{1}{|G||H|}\sum_{x,y\in G} \dot\chi(x(g-1)y+1),\quad\text{where}\quad \dot\chi(z)=\left\{\begin{array}{ll} \chi(z), & \text{if $z\in H$,}\\ 0, & \text{if $z\notin H$.}\end{array}\right.
\end{equation}
This formula generalizes induction by averaging over a superclass instead of averaging over a conjugacy class.  

As a functor, superinduction exhibits several desirable qualities, including
\begin{enumerate}
\item[(a)] It is adjoint to restriction on the space of superclass functions, so
$$\langle \SInd_H^G(\gamma),\chi\rangle=\langle \gamma,\Res_H^G(\chi)\rangle,$$
\item[(b)] It takes superclass functions of $H$ to superclass functions of $G$,
\item[(c)] The degree of $\SInd_H^G(\gamma)$ is $\gamma(1)|G|/|H|$.
\end{enumerate}
However, if $\chi$ is a supercharacter of  a subgroup $H$ of a group $G$, then $\SInd_H^G(\chi)$ is not necessarily a character of $G$.

\vspace{.25cm}

\noindent\textbf{Example.}  Let $\FF_q=\FF_2$.  Then by direct computation,
\begin{align*}
\SInd_{U_3\times U_2}^{U_5}\left(\chi^{1|2|3|4|5}\right) &= \chi^{1|2|3|4|5} + \chi^{14|2|3|5} + \chi^{1|24|3|5} + \chi^{1|2|34|5} + \chi^{15|2|3|4} + \chi^{1|25|3|4}+ \chi^{1|2|35|4} \\ & \hspace*{.5cm} + \frac{       1      }{       2      }\chi^{14|25|3} + \frac{       1      }{       2      }\chi^{14|2|35} + \chi^{15|24|3} + \frac{       1      }{       2      }\chi^{1|24|35} + \chi^{15|2|34} + \chi^{1|25|34}.
\end{align*}
The degree $\chi^{1|24|35}(1)=2$.  Since 
$$\chi^{1|24|35}=\frac{1}{c}\sum_{\chi\text{ irreducible}} \chi(1) \chi$$ 
for some $c\in \ZZ_{\geq 1}$ \cite{DI06}, we have that $\chi^{1|24|35}$ must be the sum of two linear characters of $U_5$.  Thus, the superinduced superclass function is not a character.

\vspace{.25cm}

This section explores some of the cases where superinduction turns out to be induction.   The most basic result of this nature is a consequence of the following lemma.

\begin{lemma} \label{TwistedInduction} Let $H\subseteq G$ be algebra groups, and let $\chi$ be a superclass function of $H$.  Then for $g\in G$
$$ \SInd_H^G(\chi)(g) = \frac{1}{|G|} \sum_{x \in G} \Ind_H^G(\chi)(x(g-1)+1).$$
\end{lemma}
\begin{proof}   For $g\in G$, 
$$\SInd_H^G(\chi)(g)= \frac{1}{|H||G|} \sum_{x,y\in G} \dot\chi (x(g-1)y + 1) 
= \frac{1}{|H||G|} \sum_{x,y\in G} \dot\chi(y^{-1}x(g-1)y + 1),$$
where the second equality is a substitution.  Since $y^{-1}y=1$, 
$$\SInd_H^G(\chi)(g)= \frac{1}{|G|} \sum_{x\in G} \frac{1}{|H|} \sum_{y\in G} \dot\chi(y^{-1}(x(g-1)+1)y) 
= \frac{1}{|G|} \sum_{x\in G} \Ind_H^G(\chi)(x(g-1)+1),$$
as desired.
\end{proof}

Since superinduction takes superclass functions to superclass functions, we have the following corollary.

\begin{corollary} \label{SuperinductionIsInduction}
Let $H\subseteq G$ be algebra groups.  Then $\Ind_H^G(\chi)$ is a superclass function of $G$ for every superclass function $\chi$ of $H$ if and only if $\SInd_H^G=\Ind_H^G$. 
\end{corollary}

\subsection{A left-right symmetry in supercharacter theory}\label{SectionLeftRightSymmetry}

In Section \ref{SectionPatternSupercharacters} we defined supercharacters $\chi^\lambda$ in terms of left modules $V^{-\lambda}$.  However, we obtain the same character by viewing the trace of the right module 
$$M^{-\lambda}=\CC\spanning\{v_\mu\ \mid\ \mu\in (-\lambda)U_\cP\},$$
with action given by 
$$v_\mu u=\theta(\mu(u^{-1}-1)) v_{\mu u}, \qquad\text{where $u\in U_\cP$, $\mu\in (-\lambda)U_\cP$}.$$
Thus, the supercharacter arising from the left module $V^{-\lambda}$ indexed by $-\lambda\in \fkn_\cP^*$ is the same as the supercharacter arising from the right module $M^{-\lambda}$.   In particular, if $U_\cP\subseteq U_\cR$, then for $\lambda\in \fkn_\cR^*$,
$$\Res_{U_\cP}^{U_{\cR}}(V^\lambda)=\bigoplus_{\mu\in \fkn_\cP^*} (V^\mu)^{\oplus m_\mu^\lambda}\qquad\text{if and only if} \qquad
\Res_{U_\cP}^{U_{\cR}}(M^\lambda)=\bigoplus_{\mu\in \fkn_\cP^*} (M^\mu)^{\oplus m_\mu^\lambda},$$
and if $\Ind_{U_\cP}^{U_\cR}$ is a superclass function for either left or right modules, then for $\mu\in \fkn_\cP^*$,
$$\Ind_{U_\cP}^{U_{\cR}}(V^\mu)=\bigoplus_{\lambda\in \fkn_\cR^*} (V^\lambda)^{\oplus m_\mu^\lambda}\qquad\text{if and only if} \qquad
\Res_{U_\cP}^{U_{\cR}}(M^\mu)=\bigoplus_{\lambda\in \fkn_\cR^*} (M^\lambda)^{\oplus m_\mu^\lambda}.$$
These observations will be used below to translate sufficiency conditions that are obvious for either right or left modules to the other side.

\subsection{Superinduction and induction}

Since superinduction sometimes takes characters to noncharacters, it is useful to determine when superinduction is the same as induction.   This section examines some of the cases where this occurs.

\begin{theorem}\label{AlgebraGroupSuperinductionIsInduction} Let $H$ be a subalgebra group of an algebra group $G$, and suppose
\begin{enumerate}
\item[(1)] No two superclasses of $H$ are in the same superclass of $G$,
\item[(2)] $x(h-1)+1 \in H$ for all $x\in G$, $h\in H$.
\end{enumerate}
Then for any superclass function $\chi$ of $H$, 
$$\SInd_H^G(\chi) = \Ind_H^G(\chi).$$
\end{theorem}

\begin{proof}  Let $\chi$ be a superclass function of $H$, and let
$S_1$, $S_2$,\ldots, $S_r$ be the superclasses of $H$.  By (1), there exist distinct superclasses $T_1, T_2,\ldots, T_r$ of $G$ such that $S_j\subseteq T_j$.  Note that by (\ref{SuperinductionDefinition}), 
$$\SInd_H^G(\chi)(g)=0=\Ind_H^G(\chi)(g), \qquad \text{for $g\notin T_1\cup T_2\cup\cdots \cup T_r$.}$$

WLOG suppose $g\in T_1$.  Since $\SInd_H^G(\chi)$ is constant on superclasses, we may assume $g\in S_1\subseteq G$.   If $(g-1)y+1\in H$ for $y\in G$, then by (2), $x(g-1)y+1\in H$ for all $x\in G$.  By (1), this implies $(g-1)y+1$ and $x(g-1)y+1$ are in the same superclass of $H$ and 
$$\chi(x(g-1)y+1)=  \chi((g-1)y+1).$$
If $(g-1)y+1\notin H$, then $x(g-1)y+1\notin H$.  Else, $x(g-1)y+1=h\in H$ implies
$$x^{-1}(h-1)+1=(g-1)y+1\notin H,\qquad\text{contradicting (1)}.$$
Thus,
$$\dot\chi(x(g-1)y+1)=0=\dot\chi((g-1)y+1).$$
By definition,
\begin{align*}
\SInd_H^G(\chi)(g)&=\frac{1}{|G||H|} \sum_{x,y\in G} \dot\chi(x(g-1)y+1)=\frac{1}{|G||H|} \sum_{x,y\in G} \dot\chi((g-1)y+1)\\
&=\frac{1}{|H|} \sum_{y\in G} \dot\chi(y^{-1}(g-1)y+1)=\Ind_H^G(\chi)(g),
\end{align*}
as desired.
\end{proof}

\noindent\textbf{Remarks.}  
\begin{enumerate}
\item[(a)]  The theorem still holds if we replace condition (2) with
\begin{enumerate}
\item[(2$'$)]  $(h-1)y+1 \in H$ for all $y\in G$, $h\in H$.
\end{enumerate}
\item[(b)]  For two pattern groups $U_\cP\subseteq U_\cR$, we may translate condition (2) into the condition that if $i<j$ in $\cR$ and $j<k$ in $\cP$, then $i<k$ in $\cP$.  Condition (1) is more complicated, but as we will see in Theorem \ref{SemidirectProductSuperinduction} below, $U_m\subseteq U_n$ satisfy these conditions.
\end{enumerate}

\vspace{.25cm}

Theorem \ref{SemidirectProductSuperinduction}, below, is a variant of Theorem \ref{AlgebraGroupSuperinductionIsInduction} specific to semi-direct products of pattern groups.

\begin{theorem} \label{SemidirectProductSuperinduction}
 Suppose $G = H \ltimes K$ where $G$, $H$, and $K$ are pattern groups.   If $(k-1)(h-1) = 0$ for all $h \in H$ and $k \in K$, then
$$\SInd_H^G(\chi) = \Ind_H^G(\chi) \qquad\text{for all superclass functions $\chi$ of $H$.}$$
\end{theorem}
\noindent\textbf{Remarks.} 
\begin{enumerate}
\item[(a)] The condition in Theorem \ref{SemidirectProductSuperinduction} is equivalent to the condition $i<j$ in $\cP_K$ implies $j \not< k$ in $\cP_H$ for all $k>j>i$ in $\cP_G$. 
\item[(b)] This condition may be replaced by, ``If $(h-1)(k-1)= 0$ for all $h \in H$ and $k \in K$."
\end{enumerate}

\vspace{.5cm}

We first prove a useful lemma stating that normal pattern subgroups are ``super"-normal.

\begin{lemma}\label{NormalSubgroupInvariance}
  Suppose $U_{\cP}\subseteq U_{\cR}$ are pattern groups with $U_{\cP}\triangleleft U_{\cR}$.  Then
$$x(h-1)y+1\in U_{\cP}, \qquad\text{for all $x,y\in U_{\cR}$, $h\in U_{\cP}$.}$$
\end{lemma}
\begin{proof}
Note that for $x,y\in U_\cR$ and $h\in U_\cP$,
$$1+x(h-1)y=1+xy(y^{-1}hy-1).$$
Since $U_\cP$ is normal in $U_\cR$, it suffices to show that  $1+x(h-1)\in U_\cP$ for all $x\in U_\cR$, $h\in U_\cP$.

For $x\in U_\cR$ and $h\in U_\cP$, 
\begin{align*}
x&=\prod_{i<j\text{ in $\cR$}} (1+r_{ij} e_{ij}), \qquad \text{for some $r_{ij}\in \FF_q$,}\\
h&=1+\sum_{k<l\text{ in $\cP$}} t_{kl}e_{kl}, \qquad \text{for some $t_{kl}\in \FF_q$,}
\end{align*}
so it suffices to show that $1+ (1+re_{ij})te_{kl}\in U_\cP$ for $r,t\in \FF_q$, $i<j$ in $\cR$, and $k<l$ in $\cP$.

Suppose $i<j$ in $\cR$ and $k<l$ in $\cP$.  If $j\neq k$, then for $r,t\in \FF_q$,
$$1+(1+re_{ij})te_{kl}=1+te_{kl}+re_{ij}e_{kl}=1+te_{kl}\in U_\cP.$$
If $j=k$, then
\begin{align*}
1+(1+re_{ij})te_{jl}&=1+te_{jl}+rte_{il}\\
&=1+te_{jl}+rte_{il}-rte_{jl}e_{ij}-r^2te_{il}e_{ij}\\
&=1+(1+re_{ij})te_{jl}(1-re_{ij}).
\end{align*}
Since $U_\cP$ is normal in $U_\cR$, we have $1+(1+re_{ij})te_{jl}(1-re_{ij})\in U_\cP$.
\end{proof}

\begin{proof}[Proof of Theorem \ref{SemidirectProductSuperinduction}]  We will show that $G$ and $H$ satisfy (1) and (2) of Theorem \ref{AlgebraGroupSuperinductionIsInduction}. 

(1)  Suppose $x,y\in G$ and $z\in H$ such that $1+x(z-1)y\in H$.  It suffices to show that there exist $h,h'\in H$ such that $1+x(z-1)y=1+h(z-1)h'$.  Since $G=HK$, we may write $x=kh$ and $y=h'k'$, where $h,h'\in H$, and $k,k'\in K$, so 
\begin{align*}
1+&x(z-1)y =1+kh(z-1)h'k'\\
&=1+(k-1)h(z-1)h'(k'-1)+h(z-1)h'(k'-1)+(k-1)h(z-1)h'+h(z-1)h'.
\end{align*}
Note that by Lemma \ref{NormalSubgroupInvariance}, 
$$(k-1)h(z-1)h'(k'-1),h(z-1)h'(k'-1),(k-1)h(z-1)h'\in K-1.$$ 
Since $1+x(z-1)y\in H$, we have
 $$kh(z-1)h'k'=(k-1)h(z-1)h'(k'-1)+h(z-1)h'(k'-1)+(k-1)h(z-1)h'=0$$
 and $1+x(z-1)y=1+h(z-1)h'.$

(2)  Suppose $k\in K$ and $h\in H$.  It suffices to show $1+k(h-1)\in H$ (since $h'(h-1)\in H-1$ for all $h'\in H$).  By our assumption,
$$1+k(h-1)=1+(k-1)(h-1)+(h-1)=1+(h-1)\in H,$$
as desired.
\end{proof}

\noindent\textbf{Remark.} Using the arguments in Section \ref{SectionLeftRightSymmetry}, we can replace the condition in Theorem \ref{CosetRepresentativesSuperinduction} by, ``If $(h-1)(k-1)=0$ for all $h\in H$, $k\in K$ if we switch from left modules to right modules.

\vspace{.25cm} 

\noindent\textbf{Examples.} Note that
$$U_n=U_m \ltimes U_\cP,\qquad \text{where}\qquad \cP=\xy<0cm,1.5cm> 
\xymatrix@R=.4cm@C=.3cm{ 
& & *{n} \ar @{-} [d] \\ 
& & *{n-1} \ar @{-} [d]\\
& & *{\vdots}\ar @{-} [d]\\
& & *{m+1}\\
*{1} \ar @{-} [urr] & *{2} \ar @{-} [ur] & *{\cdots} & *{m}\ar @{-} [ul]}\endxy$$
Furthermore, the assumption of Theorem \ref{SemidirectProductSuperinduction} is easily satisfied for this semi-direct product, so $\SInd_{U_m}^{U_n}=\Ind_{U_m}^{U_n}$.   

Another variant is  the semi-direct product $U_{m+n}=(U_m\times U_n) \triangleleft U_{\cP'}$, where 
$$\cP'=\xy<0cm,.5cm> 
\xymatrix@R=.4cm@C=.3cm{ 
*+{m+1} & *+{m+2} & *{\cdots} & *+{m+n}\\
*+{1} \ar @{-} [u] \ar @{-} [ur] \ar @{-} [urrr] & *+{2} \ar @{-} [ul] \ar @{-} [u] \ar @{-} [urr] & *{\cdots} & *+{m}\ar @{-} [ulll]\ar @{-} [ull] \ar @{-} [u]}\endxy$$
is the poset given by $i<j$ if $1\leq i\leq m<j\leq m+n$.   However, for $m$ or $n$ greater than 1, this semidirect product does not satisfy the hypothesis of Theorem \ref{SemidirectProductSuperinduction}, and superinduction does not, in general, give characters for these cases (see the example following Lemma \ref{TwistedInduction}).

\vspace{.25cm}

The following Theorem uses Lemma \ref{LeftRightCosetRepresentatives} to obtain a different set of groups for which superinduction is induction.

\begin{theorem} \label{CosetRepresentativesSuperinduction}
 Let $U_\cP\subseteq U_{\cR}$ be pattern groups, and let
$$I=\{u\in U_{\cR}\ \mid\  u_{ij}\neq 0 \text{ implies $i<j$ in $\cR/\cP$}\}.$$
If $(l-1)(u-1)=0$ for all $l\in I$, $u\in U_\cR$, then
$$\SInd_{U_\cP}^{U_\cR}(\chi) = \Ind_{U_\cP}^{U_\cR}(\chi) \qquad\text{for all superclass functions $\chi$ of $U_\cP$.}$$
\end{theorem}
\begin{proof}   First note that if $l,r\in I$, then by our assumption
$$lr=1+(l-1)+(r-1)+(l-1)(r-1)=1+(l-1)+(r-1)\in I.$$
Thus, $I$ is a subgroup and abelian ($l^{-1}=1+(1-l)\in I$).  

By assumption, for $l,r\in I$ and $u\in U_\cR$, 
\begin{align}
1+l(u-1)r&=r^{-1}r+r^{-1}rl(u-1)r=r^{-1}\big[(u-1)+(rl-1)(u-1)\big]r=1+r^{-1}(u-1)r \notag\\
&=r^{-1}ur. \label{LeftRightInvarianceEquation}
\end{align}
If $\chi$ a superclass function of $U_\cP$, then
\begin{align*}
\SInd_{U_\cP}^{U_\cR}(\chi)(u)&=\frac{1}{|U_\cP||U_\cR|}\sum_{x,y\in U_\cR} \dot\chi(x(u-1)y+1), && \text{by definition,}\\
&=\frac{1}{|U_\cP||U_\cR|}\sum_{l,r\in I}\sum_{h,k\in U_\cP} \dot\chi(hl(u-1)rh'+1),& & \text{by Lemma \ref{LeftRightCosetRepresentatives},}\\
&=\frac{1}{|U_\cP||U_\cR|}\sum_{l,r\in I}\sum_{h,k\in U_\cP} \dot\chi(l(u-1)r+1)\\
&=\frac{1}{|I|}\sum_{l,r\in I} \dot\chi(l(u-1)r+1)\\
&=\sum_{r\in I} \dot\chi(r^{-1}(u-1)r), && \text{by (\ref{LeftRightInvarianceEquation}),}\\
&=\Ind_{U_\cP}^{U_\cR}(\chi)(u),
\end{align*}
as desired.
\end{proof}

\noindent\textbf{Remark.} Using the arguments in Section \ref{SectionLeftRightSymmetry}, we can replace the condition in Theorem \ref{CosetRepresentativesSuperinduction} by, ``If $(u-1)(r-1)=0$ for all $r\in I$, $u\in U_\cR$ if we switch from left modules to right modules.

\vspace{.25cm}

\noindent\textbf{Example.}  Fix $n\geq 1$.  For $0\leq m\leq n$, let
\begin{equation*}
U_{(m)}  =\{u\in U_n\ \mid\  u_{1j}=0, \text{ for $j\leq m$}\}=U_{\cP_{(m)}},
\qquad\text{where}
\qquad\cP_{(m)}=\xy<0cm,2.2cm>\xymatrix@R=.25cm@C=.3cm{*={} & *+{n}\ar @{-} [d]\\ & *{\vdots}\ar @{-} [d] \\ & *+{m+1} \ar @{-} [d] \ar @{-} [dl]\\ *+{1} & *+{m}\ar @{-} [d] \\ & *+{m-1} \ar @{-} [d]\\ & *{\vdots}\ar @{-} [d]\\ & *+{2}}\endxy.
\end{equation*}
Note that 
$$U_{n-1}\cong U_{(n)} \triangleleft U_{(n-1)} \triangleleft \cdots \triangleleft U_{(1)} \triangleleft U_{(0)} = U_n,$$
and that all these groups satisfy the hypotheses of Theorem \ref{CosetRepresentativesSuperinduction} within one-another.

\section{Superinduction for $U_n$}

This section computes the superinduced characters from $U_m\subseteq U_{n}$, where
$U_n$ is the group of $n\times n$ unipotent upper-triangular matrices over $\FF_q$.   For this case, superinduction is the same as induction, and the theory gives rise to beautiful combinatorial  algorithms for adding elements to labeled set-partitions .

\subsection{An embedding of $U_{n-1}$ in $U_n$.}

Let
\begin{align*} U_{n-1}&=\{u\in U_n\ \mid\ u_{ij}\neq  0\text{ implies $i<j<n$}\},\\
\fkn_{n-1}^*&=U_{n-1}-1,\\
\fkn_n^*&=U_n-1.
\end{align*}
Note that $$U_n=U_{n-1}\ltimes U_\cP,\qquad\text{where}\qquad 
\cP=\xy<0cm,.5cm>\xymatrix@R=.75cm@C=.5cm{*={} & *={} & *+{n}\\ *{1} \ar @{-} [urr] & *{2} \ar @{-} [ur] \ar @{} [rr]^{\cdots} & & *{n-1}\ar @{-} [ul]}\endxy\ .$$
The pattern group $U_\cP$ is abelian, so its supercharacters are in bijection with $\fkn_\cP^*$.  

If $k\in U_\cP$ and $u\in U_{n-1}$ and $g=ku\in U_n$ then
$$g_{ij}=\left\{\begin{array}{ll} k_{ij}, & \text{if $j=n$,}\\ u_{ij}, & \text{if $j<n$.}\end{array}\right.$$
It follows that $(k-1)(g-1)=0$ for all $k\in U_\cP$ and $g\in U_n$.  Thus, by 
Theorem \ref{CosetRepresentativesSuperinduction}, 
$$\Ind_{U_{n-1}}^{U_n}(V^\mu)=\SInd_{U_{n-1}}^{U_n}(V^\mu), \qquad \text{for all $\mu\in \fkn_{n-1}^*$},$$
where $V^\mu$ is the supermodule corresponding to $\mu\in \fkn_{n-1}^*$.

\subsection{Induction from $U_{n-1}$ to $U_n$}

The following theorem gives a basis for the induced modules from $U_{n-1}$ to $U_n$.

\begin{theorem} \label{InducingBasis} Suppose $\mu\in \cS_{n-1}(q)$ with corresponding $U_{n-1}$-supermodule $V^\mu$.  Then
$$\Ind_{U_{n-1}}^{U_n}(V^\mu)\cong\CC\spanning\big\{ v_\lambda\ \mid\ \lambda\in \fkn^*_n,\ \lambda\big|_{U_{n-1}}\in U_{n-1}\mu\big\}.$$
\end{theorem}

\begin{proof}
For $\gamma\in \fkn^*_\cP$, define
$$e_\gamma=\frac{1}{q^{n-1}}\sum_{k\in U_\cP} \theta\circ\gamma(k^{-1}-1)k\in U_\cP.$$
Since the $e_\gamma$ are the minimal central idempotents of $\CC U_\cP$, 
$$\CC U_\cP=\CC\spanning\{e_\gamma\ \mid\ \gamma\in \fkn^*_\cP\}.$$
Then the induced module is
\begin{align*}
\Ind_{U_{n-1}}^{U_n}(V^\mu)&=\CC U_n\otimes_{\CC U_{n-1}} V^\mu\\
&= \CC\spanning\big\{e_\gamma\otimes v_\nu\ \mid\ \gamma\in \fkn_\cP^*,\ \nu\in U_{n-1}\mu\big\}.
\end{align*}
Define
$$\begin{array}{ccc}\vphi: \CC U_n\otimes_{\CC U_{n-1}} V^\mu & \longrightarrow & \CC\spanning\{ v_\lambda\ \mid\ \lambda\in \fkn^*_n, \lambda\big|_{U_{n-1}}\in U_{n-1}\mu\}\\
e_{\gamma}\otimes v_\nu & \mapsto & v_{(-\gamma)\oplus\nu},\end{array}$$
where $\gamma\in \fkn_\cP^*$ and $\nu\in U_{n-1}\mu$, and
$$\begin{array}{rccc} \gamma\oplus \nu:&\fkn_n=\fkn_\cP\oplus\fkn_{n-1} & \longrightarrow & \FF_q\\
& (k-1)+(u-1) & \mapsto & \gamma(k-1)+\nu(u-1),\end{array}$$
The map $\vphi$ is well-defined since $(-\gamma\oplus v\mu)(u-1)=v\mu(u-1)$ for all $u,v\in U_{n-1}$.  It therefore suffices to show that $\vphi$ is a $U_n$-module isomorphism.

For $k\in U_\cP$, 
\begin{equation*}
k\vphi(e_\gamma \otimes v_\nu)=kv_{(-\gamma)\oplus\nu}=\theta\big((-\gamma\oplus\nu)(k^{-1}-1)\big)v_{k(-\gamma\oplus\nu)}=\theta\big(\gamma(k-1)\big)v_{(-\gamma)\oplus\nu},\end{equation*}
since $k\gamma=\gamma$ and $k\nu(u-1)=\nu(k^{-1}(u-1))=\nu(u-1)$ for all $u\in U_{n-1}$.  On the other hand,
$$\vphi(ke_\gamma\otimes v_\nu)=\vphi(\theta(\gamma(k-1)) e_\gamma\otimes v_\nu)=\theta(\gamma(k-1))v_{(-\gamma)\oplus\nu}.$$

Suppose $u\in U_{n-1}$.  It follows from $(k-1)u=(k-1)$  for all $k\in U_\cP$ that
$$ue_\gamma u^{-1}=e_{u\gamma}$$
Thus,
$$\vphi(u e_\gamma\otimes v_\nu)=\vphi( e_{u\gamma}\otimes u v_\nu)=\theta(\nu(u^{-1}-1))\vphi( e_{u\gamma}\otimes v_{u\nu})=\theta(\nu(u^{-1}-1))v_{(-u\gamma)\oplus u\nu}.$$
On the other hand,
$$u\vphi(e_\gamma\otimes v_\nu)=u v_{(-\gamma)\oplus\nu} = \theta(\nu(u^{-1}-1))v_{u(-\gamma\oplus\nu)}= \theta(\nu(u^{-1}-1))v_{(-u\gamma)\oplus u\nu},$$
as desired.
\end{proof}

The following corollary gives a character theoretic version of Theorem \ref{InducingBasis}.  Let $\mu\in \fkn_{n-1}^*$ and $\lambda\in \fkn_n^*$.  Then $\lambda$ is \emph{left minimal over} $\mu$ if
\begin{enumerate}
\item[(1)] $\lambda\big|_{U_{n-1}}=\mu$, and 
\item[(2)] $\lambda_{in},\lambda_{jn}\in \FF_q^\times$ with $i<j$ implies $\lambda_{ik}\neq 0$ for some $k>j$.
\end{enumerate}

\begin{corollary}  \label{SuperCharacterSum} Suppose $\mu\in \cS_{n-1}(q)$ with corresponding $U_{n-1}$-supercharacter $\chi^\mu$.  Then
$$\Ind_{U_{n-1}}^{U_n}(\chi^\mu)=\sum_{\text{$\lambda$ left minimal}\atop \text{over $\mu$}}\chi^\lambda,$$
where the supercharacters $\chi^\lambda$ are not necessarily distinct.  
\end{corollary}
\begin{proof}  By Theorem \ref{InducingBasis}, $V^\lambda$ is a submodule of $\Ind_{U_{n-1}}^{U_n}(V^\mu)$ if and only if 
$$\lambda\big|_{U_{n-1}}\in U_{n-1}\mu.$$
For each such $\lambda$,
$$V^\lambda=\CC\spanning\{v_{\lambda'}\ \mid\ \lambda'\in U_n\lambda\}.$$
  Note that each module $V^\lambda$ has at least one basis vector $v_{\lambda_0}$ such that $\lambda_0\big|_{U_{n-1}}=\mu$.  Furthermore, there is a unique $\lambda_0$ that has a minimal number of nonzero entries in the $n$th column.  This uniquely defined representative of $V^\lambda$ is left minimal over $\mu$.  

Conversely, every $\lambda$ that is left minimal over $\mu$ satisfies $\lambda\big|_{U_{n-1}}\in U_{n-1}\mu$ and $\lambda$ has the minimal number of nonzero entries in the $n$th column among all $\lambda_0\in U_n\lambda$ such that $\lambda_0\big|_{U_{n-1}}=\mu$.
\end{proof}

Define two products on labeled set partitions by the following rule.  For a labeled set partition $\mu$ and $i,k\in \ZZ_{\geq 1}$, let
\begin{align} \mu *_i \{k\} &= \left\{ \begin{array}{ll} 
\mu\cup\{k\},&\text{if } i=k, \\ 
q(\mu*_{i+1} \{k\}), &\text{if there is }l>k\text{ with } \mu_{il}\neq 0, \\ 
 \mu\big|_{i\larc{\mu_{ik}}k\mapsto i\mid k}*_{i+1}\{k\},&\text{if } \mu_{ik}\neq 0,\\ 
\displaystyle\mu*_{i+1}\{k\}+\sum_{t \in \FF_q^\times} \mu\big|_{i\larc{\mu_{ij}}j\mapsto i\larc{t}k}*_{i+1}\{j\} ,&\text{if there is }j<k\text{ with } \mu_{ij}\neq 0,\\ 
\displaystyle  \mu*_{i+1}\{k\}+\sum_{t \in \FF_q^\times} \mu\big|_{i\mid\mapsto i\larc{t}k},&\text{otherwise,} \end{array}\right.\\
\intertext{and for $j,l\in \ZZ_{\geq 1}$, let}
\{j\} \ast_l \mu&=\left\{\begin{array}{ll} \mu\cup\{l\}, & \text{if $j=l$,}\\ 
q (\{j\} *_{l-1} \mu), & \text{if there is $i<j$ with $\mu_{il}\neq 0$,}\\
\dd \{j\} \ast_{l-1} \mu\big|_{j\larc{\mu_{jl}}l\mapsto j\mid l}, & \text{if $\mu_{jl}\neq 0$,}\\
\dd \{j\}\ast_{l-1}\mu+\sum_{t\in \FF_q^\times} \{k\} \ast_{l-1} \mu\big|_{k\larc{\mu_{kl}}l\mapsto j\larc{t} l}, & \text{if there is $k>j$ with $\mu_{kl}\neq 0$,}\\
\dd \{j\} \ast_{l-1} \mu + \sum_{t\in \FF_q^\times} \mu\big|_{j\mid\mapsto j\larc{t}l}, & \text{otherwise,}\end{array}\right.
\end{align}
where the notation $\mu\big|_{j\larc{\mu_{jl}}l\mapsto j\mid l}$ indicates replacing $j\larc{\mu_{jl}}l$ in $\mu$ with $j\mid l$, and leaving everything else in $\mu$ the same.  For example,
$$\{1\larc{a}3\larc{b}6\larc{c}7\mid 2\larc{d}5\mid 4\}\big|_{3\larc{b}6\mapsto 3\mid 6}=\{1\larc{a}3\mid 2\larc{d}5\mid 4\mid 6\larc{c}7\}.$$

We will extend these product to supercharacters by the conventions
\begin{align*}
\chi^\mu\ast_i \chi^{\{k\}} &=\sum_{\lambda} c_{\mu k}^\lambda \chi^\lambda,\qquad\text{if} \qquad \mu\ast_i \{k\}=\sum_{\lambda}c_{\mu k}^\lambda\lambda \\ 
\chi^{\{j\}}\ast_l\chi^\mu &=\sum_\lambda c_{j\mu}^\lambda\chi^\lambda,\qquad\text{if}\qquad {\{j\}}\ast_l\mu = \sum_\lambda c_{j\mu}^\lambda \lambda.
\end{align*}

\begin{theorem} \label{InductionAlgorithm}
Let $\mu$ be an $\FF_q$-labeled set partition of $\{1,2,\ldots, n-1\}$.  Then
$$\Ind_{U_{n-1}}^{U_n}(\chi^\mu)=\chi^\mu\ast_1 \chi^{\{n\}}.$$
\end{theorem}
\begin{proof}  We induct on $n$, where the base case is clear.
For $\mu\in \fkn_{n-1}^*$, let
$$L_\mu=\{\lambda\in \fkn^*_n\ \mid\ \lambda\text{ is left minimal over $\mu$}\}.$$
By Corollary \ref{SuperCharacterSum},
$$\Ind_{U_{n-1}}^{U_n} (\chi^\mu)=\sum_{\lambda\in L_\mu} \chi^\lambda.$$
We will show that
$$\Ind_{U_{n-1}}^{U_n}(\chi^\mu)=\left\{\begin{array}{ll} \displaystyle\mu*_{2}\{n\}+\sum_{t \in \FF_q^\times} \mu\big|_{1\larc{\mu_{1j}}j\mapsto 1\larc{t}n}*_{2}\{j\} ,&\text{if there is }j<n\text{ with } \mu_{1j}\neq 0,\\ 
\displaystyle  \mu*_{2}\{n\}+\sum_{t \in \FF_q^\times} \mu\big|_{1\mid\mapsto 1\larc{t}n},&\text{otherwise.} \end{array}\right.$$

Suppose $\mu$ has no nonzero entry in the first row.  Then $\lambda\in L_\mu$ either satisfies
\begin{enumerate}
\item[(a)] $\lambda_{1n}\neq 0$ and $\lambda$ has no other nonzero entry in the $n$th column,
\item[(b)] $\lambda_{1n}=0$ and $\lambda$ is left minimal over $\mu$ when we restrict to the pattern groups not including the first row.
\end{enumerate}
Thus, in this case,
\begin{align*}
\Ind_{U_{n-1}}^{U_n} (\chi^\mu)&=\sum_{\lambda\in L_\mu\atop \lambda_{1n}=0} \chi^\lambda+\sum_{\lambda\in L_\mu\atop \lambda_{1n}\in \FF_q^\times} \chi^\lambda\\
&=\Ind_{U_{n-2}'}^{U_{n-1}'}(\chi^\mu)+\sum_{t \in \FF_q^\times} \chi^{\mu\big|_{1\mid\mapsto 1\larc{t}n}},
\end{align*}
where
\begin{align*}
U_{n-1}' &=\{u\in U_n\ \mid\ u_{1j}=0, 1< j\leq n\}\\
U_{n-2}' &=\{u\in U_{n-1}'\ \mid\ u_{ik}=0, k=n\}.
\end{align*}
However, by induction on the size of $\{1,2,\ldots, n\}$, we have that
$$\Ind_{U_{n-2}'}^{U_{n-1}'}(\chi^\mu)=\chi^\mu\ast_2 n,$$
as desired.

Suppose that $\mu_{1j}\neq 0$ for some $1<j<n$.  Then
\begin{equation}\label{SplittingOne}
\sum_{\lambda\in L_\mu}\chi^\lambda=\sum_{\lambda\in L_\mu\atop \lambda_{1n}=0} \chi^\lambda+\sum_{\lambda\in L_\mu\atop\lambda_{1n}\in \FF_q^\times} \chi^\lambda.
\end{equation}
Consider the first sum, and  let $\mu'$ be the same as $\mu$ except with $\mu'_{1j}=0$.  Then
$$\Ind_{U_{n-2}'}^{U_{n-1}'}(\chi^{\mu'})=\sum_{\lambda' \text{ left minimal}\atop\text{over $\mu'$}} \chi^{\lambda'}=\chi^{\mu'}\ast_2 \chi^{\{n\}},$$
by induction.  Since $\lambda_{1n}=0$, any row and column reducing will not affect the first row, it follows that
$$\chi^\mu\ast_2 \chi^{\{n\}}=\sum_{\lambda\in L_\mu\atop \lambda_{1n}=0} \chi^\lambda.$$

Consider the second sum on the RHS in (\ref{SplittingOne}), and let $\lambda\in L_\mu$ such that $\lambda_{1n}\neq 0$.  Since $\lambda_{1n}, \lambda_{1j}\in \FF_q^\times$, we can reduce 
\begin{equation}\label{BigReduction}
\left(\begin{array}{cccc|c|cc} 
& &  & & \mu_{1j} &  & \lambda_{1n}\\  
& &  & &  & &\lambda_{2n} \\
& &  & & & & \vdots \\
& &  & & & &\lambda_{n-1,n} \\
& &  & & & & 0\\
& &  & & & & \end{array}\right)\qquad\text{to}\qquad \lambda'=\left(\begin{array}{cccc|c|cc} 
& &  & & 0 &  & \lambda_{1n}\\  
& &  & & -\mu_{1j}\lambda_{1n}^{-1}\lambda_{2n} & & 0 \\
& &  & & \vdots & & 0\\
 & &  & & -\mu_{1j}\lambda_{1n}^{-1}\lambda_{n-2,n} & & \vdots \\
& &  & & 0 & & 0 \\
 & &   & & & & 0
  \end{array}\right).\end{equation}
Let $\lambda'$ denote this reduction applied to $\lambda$, so
$$\sum_{\lambda\in L_\mu\atop\lambda_{1n}\in \FF_q^\times} \chi^\lambda=\sum_{\lambda\in L_\mu\atop\lambda_{1n}\in \FF_q^\times} \chi^{\lambda'},$$
Note that this has the effect of replacing $1\larc{\mu_{1j}} j$ by $1\larc{\lambda_{1n}} n$, and no further row or column operations will affect the first row or last column.  

If $\mu_{2k}\neq 0$ for some $k>j$, then we can reduce
$$\lambda'=\left(\begin{array}{ccccccc} 
& &  & & 0 &  & \lambda_{1n}\\  \hline
& &  & & -\mu_{1j}\lambda_{1n}^{-1}\lambda_{2n} & \mu_{2k} & 0 \\ \hline
& &  & & \vdots & & 0\\
 & &  & & -\mu_{1j}\lambda_{1n}^{-1}\lambda_{n-2,n} & & \vdots \\
& &  & & 0 & & 0 \\
 & &   & & & & 0
  \end{array}\right)\ \text{to}\ \lambda''=\left(\begin{array}{ccccccc} 
& &  & & 0 &  & \lambda_{1n}\\  \hline
& &  & & 0 & \mu_{2k} & 0 \\ \hline
& & & & -\mu_{1j}\lambda_{1n}^{-1}\lambda_{3n}  & & 0\\
& &  & & \vdots & & 0\\
 & &  & & -\mu_{1j}\lambda_{1n}^{-1}\lambda_{n-2,n} & & \vdots \\
& &  & & 0 & & 0 \\
 & &   & & & & 0
  \end{array}\right)$$
That is, in this case, every possible value of $\lambda_{2n}$ leads to the same reduced matrix.  In particular,
$$\sum_{\lambda\in L_\mu\atop\lambda_{1n}\in \FF_q^\times} \chi^{\lambda'}=q\sum_{\lambda\in L_\mu\atop\lambda_{1n}\in \FF_q^\times,\lambda_{2n}=0} \chi^{\lambda''},$$
and any further reductions will not affect the second row.  

If $\mu_{2k}=0$ for all $j<k<n$, then we proceed reducing as before, replacing $1$ by $2$ and $n$ by $j$.  Combining these steps we obtain that
$$\sum_{\lambda\in L_\mu\atop\lambda_{1n}\in \FF_q^\times} \chi^\lambda=\sum_{t\in \FF_q^\times} \mu\big|_{i\larc{\mu_{ij}}j\mapsto i\larc{t}k}*_{2}j,$$
as desired.
\end{proof}

\noindent\textbf{Example.}  If $\FF_q=\FF_2$, then we obtain a Pieri-type product on set partitions using Theorem \ref{InductionAlgorithm} as follows. Let $\{1,2,\ldots,n-1\}=A\cup B$ with $A\cap B=\emptyset$ and $B=\{b_1<b_2<\cdots<b_m\}$.  For $\lambda=\{\lambda_1\mid\lambda_2\mid\cdots\mid\lambda_k\}$ of $A$, define recursively
\begin{align*}
\lambda\ast_{i} &\{b_1\larc{}b_2\larc{}\cdots \larc{} b_m\}\\ 
&=2^{|\{j\larc{}k\in \lambda\ \mid\ i\leq j<b_1<k\}|} \{\lambda\mid b_1\larc{}b_2\larc{}\cdots \larc{} b_m\}+\sum_{a\mid\in \lambda\atop i<a<b_1} \lambda\bigg|_{a\mid\mapsto a\larc{}b_1\larc{}b_2\larc{}\cdots\larc{} b_m\mid}\\
&\hspace*{.25cm}+\sum_{a_1\larc{}a_2\larc{}\cdots\larc{}a_r\mid\in \lambda\atop i<a_1<a_2<b_1} 2^{|\{j\larc{}k\in \lambda\ \mid\ i\leq j<a_1<b_1<k\}|}   \lambda\bigg|_{a_1\larc{}a_2\larc{}\cdots\larc{}a_r\mid\mapsto a_1\larc{}b_1\larc{}\cdots\larc{t}b_m\mid} \ast_{a_1+1}  \{a_2\larc{}\cdots\larc{}a_r\}.
\end{align*}
For example,
\begin{align*}
\{1\larc{}4&\larc{}6 \big| 2 \big| 3\larc{}5\} \ast_{1}\{7\}=\{1\larc{}4\larc{}6 \big| 2 \big| 3\larc{}5\big|7\}\\
&+ \big(\{1\larc{}4\larc{}6\larc{}7 \big| 2 \big| 3\larc{}5\}+\{1\larc{}4\larc{}6 \big| 2\larc{} 7 \big| 3\larc{}5\}+\{1\larc{}4\larc{}6 \big| 2 \big| 3\larc{}5\larc{}7\}\big)\\
&+\{1\larc{} 7 \big| 2 \big| 3\larc{}5\}\ast_{2}\{4\larc{}6\}+ \{1\larc{}4\larc{} 7 \big| 2 \big| 3\larc{}5\}\ast_{5} \{6\}
+\{1\larc{}4\larc{}6 \big| 2 \big| 3\larc{} 7\}\ast_{4} \{5\},
\end{align*}
where
\begin{align*}
\{1\larc{} 7 \big| 2 \big| 3\larc{}5\} \ast_{2} \{4\larc{}6\}&= 2\{1\larc{} 7 \big| 2 \big| 3\larc{}5\big|4\larc{}6\}+\{1\larc{} 7 \big| 2\larc{}4\larc{}6 \big| 3\larc{}5\}\\
\{1\larc{}4\larc{} 7 \big| 2 \big| 3\larc{}5\}\ast_{5} \{6\}&=\{1\larc{}4\larc{} 7 \big| 2 \big| 3\larc{}5\big|6\}+\{1\larc{}4\larc{} 7 \big| 2 \big| 3\larc{}5\larc{} 6\}\\
\{1\larc{}4\larc{}6 \big| 2 \big| 3\larc{} 7\}\ast_{4} \{5\}&=2\{1\larc{}4\larc{}6 \big| 2 \big| 3\larc{} 7\big| 5\}.
 \end{align*}
 Thus,
 \begin{align*}\Ind_{U_6}^{U_7}(\chi^{\{1\larc{}4\larc{}6 \mid 2 \mid 3\larc{}5\}})&= \chi^{\{1\larc{}4\larc{}6 \mid 2 \mid 3\larc{}5\mid 7\}}+\chi^{\{1\larc{}4\larc{}6\larc{}7 \mid 2 \mid 3\larc{}5\}}+\chi^{\{1\larc{}4\larc{}6 \mid 2\larc{} 7 \mid 3\larc{}5\}}+\chi^{\{1\larc{}4\larc{}6 \mid 2 \mid 3\larc{}5\larc{}7\}}\\
& +2\chi^{\{1\larc{} 7 \mid 2 \mid 3\larc{}5\mid 4\larc{}6\}}+ \chi^{\{1\larc{} 7 \mid 2\larc{}4\larc{}6 \mid 3\larc{}5\}} +\chi^{\{1\larc{}4\larc{} 7 \mid 2 \mid 3\larc{}5\mid 6\}}+\chi^{\{1\larc{}4\larc{} 7 \mid 2 \mid 3\larc{}5\larc{} 6\}}\\
&+2\chi^{\{1\larc{}4\larc{}6 \mid 2 \mid 3\larc{} 7\mid 5\}}.
\end{align*}
There is a similar formula for arbitrary $q$ obtained by summing over all $\sum_{t\in\FF_q^\times} a\larc{t}b$ whenever we add an arc.   

\begin{corollary} Let $\mu\in \cS_{m}(q)$.  Then
$$\Ind_{U_m}^{U_n}(\chi^\mu)=\chi^\mu\ast_1 \chi^{\{m+1\}}\ast_1 \chi^{\{m+2\}}\ast_{1} \cdots\ast_{1} \chi^{\{n\}}.$$
\end{corollary}

The paper \cite{TVe07} proves the following Theorem.

\begin{theorem}  For $\lambda\in \cS_n(q)$, 
$$\Res_{U_{n-1}}^{U_n}(\chi^\lambda)=\left\{\begin{array}{ll} \chi^{\{i\}}\ast_n \chi^\lambda, & \text{if $\lambda_{in}\neq 0$,}\\ \chi^{\lambda\big|_{U_{n-1}}}, & \text{otherwise.}\end{array}\right.$$
\end{theorem}

By applying Frobenius reciprocity, we obtain the following corollary.

\begin{corollary}  Suppose $\mu\in \cS_{n-1}(q)$ and $\lambda\in \cS_{n}(q)$. 
\begin{enumerate}
\item[(a)] If $\lambda_{in}\neq 0$ for some $1\leq i\leq n-1$, then
$$\langle \chi^\lambda, \chi^\mu\ast_1 \chi^{\{n\}}\rangle=\langle \chi^{\{i\}}\ast_n\chi^\lambda,\chi^\mu\rangle.$$
\item[(b)] If  $\lambda_{in}= 0$ for all $1\leq i\leq n-1$, then 
$$\langle \chi^\lambda, \chi^\mu\ast_1 \chi^{\{n\}}\rangle=\left\{\begin{array}{ll} 1, & \text{if $\lambda\big|_{U_{n-1}}=\mu$,}\\ 0, & \text{otherwise.}\end{array}\right.$$
\end{enumerate}
\end{corollary}

\subsection{An alternate embedding of $U_{n-1}$}

The paper \cite{TVe07} uses  a different embedding of $U_{n-1}$ into $U_n$ (obtained by removing the first row rather than the last column).  This embedding no longer satisfies the conditions of Theorem \ref{SemidirectProductSuperinduction} for its left modules.  However, by Section \ref{SectionLeftRightSymmetry}, we may instead consider right modules.  In this case, Theorem \ref{SemidirectProductSuperinduction} applies and we get the same sequence of results with left and right reversed.  In fact, we have the following corollary.

\begin{corollary}  Let $U_{n-1}\subseteq U_n$ be the embedding obtained by setting the first row of $U_n$ equal to zero.  Then for $\mu$ an $\FF_q$-labeled set-partition of $\{2,3,\ldots,n\}$, 
$$\Ind_{U_{n-1}}^{U_n}(\chi^\mu)=\chi^{\{1\}}\ast_n\chi^\mu.$$
\end{corollary}

In particular, unlike in the symmetric group representation theory, the decomposition of induced characters depends on the embedding of $U_{n-1}$ into $U_n$.

\end{document}